\documentclass{amsart}
\usepackage{amssymb,amsthm,amsmath,amstext,enumitem,mathrsfs,bm}
\usepackage[utf8]{inputenc}
\usepackage{booktabs}                       
\usepackage{calligra,mathrsfs}
\usepackage{textcomp}

\usepackage[left=2.7cm,top=2.7cm,right=2.7cm,bottom=2.7cm]{geometry}
\usepackage[all]{xy}

\title{Ulrich ranks of Veronese varieties and equivariant instantons}

\author{Daniele Faenzi}
\address{Daniele Faenzi, 
Université Bourgogne Europe, CNRS, IMB UMR 5584, F-21000 Dijon, France
}
\email{daniele.faenzi@u-bourgogne.fr}

\author{Victor Pretti}
\address{Victor Pretti, Instituto de Matem\'atica e Estat\'istica,
Universidade de S\~ao Paulo,
1010 Rua do Mat\~ao,
S\~ao Paulo, Brazil}
\email{univ-pretti.w3qpo@simplelogin.com}

\theoremstyle{plain}

\newtheorem{theorem}{Theorem}
\newtheorem*{conjecture*}{Conjecture}

\newtheorem{lemma}[theorem]{Lemma}

\theoremstyle{definition}

\newtheorem{remark}[theorem]{Remark}

\DeclareMathOperator{\Ext}{Ext}

\DeclareMathOperator{\Hom}{Hom}

\DeclareMathOperator{\coker}{coker}
\DeclareMathOperator{\Pic}{Pic}

 \def\CC{{\mathbb{C}}}  \def\NN{{\mathbb{N}}} \def\PP{{\mathbb{P}}}    \def\ZZ{{\mathbb{Z}}} 
\def\SL{\mathrm{SL}}

\def\cE{{\mathcal{E}}}
\def\cF{{\mathcal{F}}}
\def\cG{{\mathcal{G}}}

\def\cO{{\mathcal{O}}}

\def\cT{{\mathcal{T}}}

\def\cU{{\mathcal{U}}}

\def\rI{{\mathrm{MI}}}
\def\ad{{\mathrm{ad}}}
\def\rP{{\mathrm{LP}}}
\def\rSI{{\mathrm{MI_{Sp}}}}
\def\rOI{{\mathrm{MI_{SO}}}}
\def\rO{{\mathrm{M_{SO}}}}
\def\SO{{\mathrm{SO}}}
\def\rk{{\mathrm{rk}}}
\def\Ur{{\mathrm{Ur}}}
\def\Sp{{\mathrm{Sp}}}

\usepackage[usenames,dvipsnames]{xcolor}

\usepackage{hyperref}

\begin{document}


\begin{thanks}{D. F. partially supported by FanoHK ANR-20-CE40-0023, SupToPhAG/EIPHI ANR-17-EURE-0002, Bridges ANR-21-CE40-0017, Capes/COFECUB Ma 926/19.
V.P. partially supported by CAPES 88887.647097/2021-00, FAPESP 2022/12883-3. Both authors thank Departamento de Matemática de Belo Horizonte for the warm hospitality in October 2023.
}\end{thanks}

\begin{abstract}
We construct Ulrich bundles on Veronese threefolds of arbitrary degree as generic deformations of symmetric squares of equivariant instanton bundles on the projective space, thus classifying the rank of Ulrich bundles on such varieties and proving a conjecture of Costa and Miró-Roig.
\end{abstract}

\maketitle

\section{Introduction}

Given an $n$-dimensional closed subscheme $X \subset \PP^N$ polarized by a very ample divisor class $D$, an Ulrich sheaf of $X$ is a non-zero coherent sheaf $\cU$ such that $H^*(\cU(-t D))=0$ for $1 \le t \le n$.
Ulrich sheaves attracted a lot of attention in view of their multiple links to other topics such as Boij-Söderberg theory, Chow forms, matrix factorisations and so on.
As general reference on Ulrich bundles we use \cite{costa-miro-llopis}.
We define the set of Ulrich ranks of $X$ as
\[
\Ur(X) = \{r \in \NN \mid \mbox{there exists an Ulrich sheaf $\cU$ on $X$ with $\rk(\cU)=r$}\}.
\]

In general, it is very difficult to determine $\Ur(X)$, except when $n=1$ where one easily finds $\Ur(X)=\NN^*$. A fundamental conjecture of \cite{eisenbud-schreyer:chow} states that $\Ur(X)$ should be always non empty, while
the main point of \cite{blaser-eisenbud-schreyer} is to compute the lower bound of 
$\Ur(X)$, called the Ulrich complexity of $X$. 
For a few surfaces, we understand $\Ur(X)$ completely, for instance we know $\Ur(X)  = 2\NN^*$ for K3 surfaces of Picard number one, see \cite{K3,farkas-aprodu-ortega}; for the $d$-th Veronese surface, we know $\Ur(X)=2\NN^*$ if $d$ is even and $\Ur(X)=\NN^* \setminus \{1\}$ if $d$ is odd, see \cite{costa-miro:ulrich-veronese-surfaces}.
For $n \ge 3$ and for an arbitrary polarization, typically we have no clue of what $\Ur(X)$ should be, even though some results are available for minimal polarizations, e.g. \cite{ciliberto-flamini-knutsen:del_pezzo}.

In this note, we focus on Veronese varieties of dimension $n \ge 3$ and mostly on $n=3$, where the determination of $\Ur(X)$ is open. Our main result gives a complete answer for Veronese threefolds of arbitrary degree, answering a conjecture of Costa and Miró-Roig and giving the first example of varieties of dimension $n > 2$ where $\Ur(X)$ is completely understood, relatively to an arbitrary polarization. Actually $\Ur(X)$ is easy to determine also for $d$-th Veronese $n$-folds when $n!$ divides $d$, as Lemma \ref{n!} below shows, but when $d$ mod $n!$ is arbitrary this gets much harder. Here we look at $n=3$, the most interesting cases being when $d$ is congruent to $1,3$ or $5$ mod $6$.

Let us formulate the precise result. Given integers $d,n \ge 1$ we set $N_d^n={n+d \choose n}-1$ and consider the $d$-th Veronese $n$-fold $X^n_d$ in $\PP^{N_d^n}$. We omit $n$ from the notation when $n=3$. 

\begin{theorem}\label{main-1}
Let $d \ge 2$ and let $X_d$ be the $d$-th Veronese threefold in $\PP^{N_d}$. Set $\bar d \in \{0,\ldots,5\}$ for the remainder of the divison of $d$ by $6$. We have:
\begin{enumerate}[label=\roman*)]
    \item \label{0} if $\bar d=0$, then $\Ur(X_d)=6\NN^*$;
    \item \label{1} if $\bar d \in \{1,5\}$, then $\Ur(X_d)=\NN^* \setminus \{1\}$;
    \item \label{2} if $\bar d \in \{2,4\}$, then $\Ur(X_d)=2\NN^*$;
    \item \label{3} if $\bar d = 3$, then $\Ur(X_d)=3\NN^*$.
\end{enumerate}
\end{theorem}

This proves \cite[Conjecture 1.1]{costa-miro:instanton-ulrich}, where of course $r=1$ has to be excluded from the statement. Indeed, the conjecture states that, for $d\ge2$, there should exist an Ulrich bundle of rank $r \ge 2$ on $X_d$ if and only if  $r(d^2-1)\equiv 0$ mod $6$,
and, in cases \ref{0}, \ref{1}, \ref{2}, \ref{3}, $d^2-1$ is congruous to, respectively, $-1,0,3,2$ mod $6$, so 
$r$ should be, respectively,
a multiple of $6$, $1$, $2$, $3$.

Let us say a word on the proof.
First, non-existence outside the range mentioned in the theorem is easily dealt with by looking at the Hilbert polynomial of an Ulrich bundle on $X_d$. So the point is about constructing Ulrich bundles in the prescribed ranks. 

About existence, one first notes that Ulrich bundles of rank $r$ are available when $n!$ divides $r$, due to a simple and smart construction appearing in \cite[Theorem 6.1]{eisenbud-schreyer:betti}. On the other hand, finding Ulrich bundles of rank $r < n!$ may be challenging and actually impossible for low $r$. However Costa and Miró-Roig observed that, if $3$ divides $d^2-1$ and $\cE$ is a sufficiently general instanton bundle of rank $2$ and charge  $\frac 13(d^2-1)$, then $\cE(2(d-1))$ is an Ulrich bundle on $X_d$. This relies on a difficult result of Hartshorne and Hirschowitz, actually the main result of \cite{hartshorne-hirschiwitz}, to the effect that $\cE$ has natural cohomology, meaning that for all $t \in \ZZ$ there is at most one value of $i \in \NN$ giving $H^i(\cE(t)) \ne 0$. This, combined with Riemann-Roch, implies the vanishing of cohomology required for $\cE(2(d-1))$ to be Ulrich.

So in order to complete the classification, the main obstacle is constructing Ulrich bundles of rank 3 in the prescribed range, which is precisely the main goal of this note. The rather surprising fact is that, instead of searching for general bundles relying on the usual techniques based on constructing bundles from curves (typically unions of lines and conics) combined with deformation theory, we achieve this using very special bundles, in some sense the most special we could think of, namely symmetric squares of equivariant instanton bundles for the action of $\SL_2$ on $\PP^3$ by binary cubics.

These instantons were examined in the framework of spaces of equivariant matrices of constant rank in \cite{boralevi-faenzi-lella}. They had been previously constructed and classified in \cite{faenzi:SL2} and actually much earlier, to the first author's astonishment, see \cite{bor-segert}.
The outcome of the classification is that, for any integer $m \ge 1$, there is one and only one instanton bundle $\cE_m$ of rank $2$ on $\PP^3$ which is $\SL_2$-equivariant for this action and having $c_2(\cE)={m + 1 \choose 2}$.
These are precisely the numerical invariants needed to get Ulrich bundles after taking symmetric squares.

To explain how these sheaves are useful in the present setting, let us state a second result. We write $\rI(r,k)$ for the coarse moduli space of rank-$r$ stable instanton bundles having second Chern class $k H^2$, where $H$ is the hyperplane class of $\PP^3$. 
According to \cite{jardim-verbitsky:trihyperkahler}, the space $\rI(2,k)$ is a smooth variety which is moreover irreducible by \cite{tikhomirov:odd, tikhomirov:even}.
We recall that Le Potier defined in \cite{le_potier:yang-mills} an effective divisor that we denote by $\rP(k) \subset \rI(2,k)$ consisting of instanton bundles $\cE$ such that $H^*(S^2 \cE(-2)) \ne 0$.
We also consider $\rSI(r,k)$ and $\rOI(r,k)$, the moduli spaces of symplectic and orthogonal instantons, namely pairs $([\cF],[\eta])$, where $[\cF]$ lies in $\rI(r,k)$ and $[\eta]$ is the proportionality class of $\eta : \cF \to \cF^\vee$, a skew-symmetric or symmetric duality. Hence actually $\rI(2,k) \simeq \rSI(2,k)$ for any $k \ge 1$.
The following theorem can be thought of as an incarnation of the classical isomorphism of Lie algebras $\mathfrak{sl}_2(\CC) \simeq \mathfrak{so}_3(\CC)$.

\begin{theorem}\label{main-2}
For any $k \ge 1$, the assignment $[\cE] \mapsto [S^2 \cE]$ gives an étale map onto an open subset:
$$\varphi : \rI(2,k) \setminus \rP(k) \to \rOI(3,4k).$$ 
For $e \in \{0,1,2\}$ and $h \ge 0$, with $(h,e) \ne (0,0)$, write $d=6h+2e+1$, $m=3h+e$, $k = {m+1 \choose 2}$. Then, for $[\cE]$ sufficiently general in  $\rI(2,k)$, $S^2 \cE(2(d-1))$ is an Ulrich bundle on $X_d$.
\end{theorem}

The crux of the argument is that, when $\cE$ is the $\SL_2$-equivariant instanton bundle $\cE_m$ mentioned above, we may show that $\cF_m = S^2 \cE_m$ satisfies $H^*(\cF_m(d-2))=0$. This uses a remarkable sheafified minimal graded free $\SL_2$-equivariant resolution of $\cE_m$, obtained in \cite{faenzi:SL2} and recovered in \cite{boralevi-faenzi-lella}, that looks like:
\begin{equation} \label{resolution}
0 \to \cO_{\PP^3}(-2m-1) \to V_{3m}\otimes \cO_{\PP^3}(-m-1) \to V_{3m+1} \otimes \cO_{\PP^3}(-m) \to \cE_m \to 0.
\end{equation}
Here, for $p \ge 0$, $V_p=S^p V_1$ is the irreducible $\SL_2$-representation of weight $p$.

With this being done, one observes that $H^*(\cF_m(-d-2))=0$ follows from Serre duality simply because $\cF_m$ is orthogonal. Then, choosing $[\cE] \notin \rP(k)$ in a neighborhood of $[\cE_m]$ will turn out in $S^2 \cE(2(d-1))$ being an Ulrich bundle.
In the very last passage, we use the irreducibility of $\rI(2,k)$.
We could actually bypass it by showing $H^*(\cF_m(-2))=0$, which we checked for several values of $m$ but failed to prove in general. Note that $\cE_m$ itself is far from having natural cohomology, as the resolution \eqref{resolution} shows, but $S^2 \cE_m$ should have natural cohomology, at least we know it does for low values of $m$.

\begin{remark}
For $n \ge 4$ and arbitrary $d$, it seems difficult to determine $\Ur(X^d_n)$. Here are some comments.
\begin{enumerate}[label=\roman*)]
\item According to arithmetic conditions on $d$ and $n$, we may determine $\Ur(X^n_d)$, for instance if $n!$ divides $d$ then $\Ur(X^n_d)=n! \NN^*$, see Lemma \ref{n!}.
\item For most values of $d$ mod $n!$, we may not have a full picture of $\Ur(X^n_d)$.
For instance, for $n=5$ and $d=2$, for any Ulrich bundle $\cU$ of rank $r$ the discussion of Lemma \ref{n!} gives $8 \mid r$. Letting $\Sp_6$ act linearly on $\PP^5$, 
we may consider the $\Sp_6$-bundle $\cF$ associated with the irreducible representation of maximal weight $\varpi_2+\varpi_3$, where $\varpi_1,\varpi_2,\varpi_3$ are the fundamental weights for $\Sp_6$. Then Borel-Weil ensures that $\cU=\cF(1)$ is an Ulrich bundle on $X^5_2$. We have $\rk(\cU)=16$. But for any odd number $k < 15$, we don't know if $X^5_2$ supports Ulrich bundles of rank $8k$.
\item For $n \ge 4$ and $d \ge 2$, according to recent work of Lopez and Raychaudhury, see \cite{lopez-raychaudhury:veronese}, $X^n_d$ does not carry Ulrich bundles of rank $\le 3$. This indicates that the situation for $X^3_d$ is very special. 
\end{enumerate}
\end{remark}

\section{Proof of the main results}

Let us go through the proof of the results announced in the introduction. In the next subsection we recall some down-to-earth arithmetics about Ulrich bundles and classify Ulrich ranks on Veronese varieties in the simplest situation.
Then we show a rather unexpected lemma stating that the $\SL_2$ equivariant instantons considered here, in spite of being far from having natural cohomology, have a symmetric square that (almost) does. In the following subsection we interpret the symmetric square in terms of moduli spaces and show the fact bundles obtained this way fill an open dense subset of a component of their moduli space. This is surprising at first sight, but turns out to be very natural by looking at adjoint bundles. The last subsection summarizes the proof on Theorem \ref{main-1}.

\subsection{Reminder on arithmetics of Ulrich bundles}

Let $\cU$ be an Ulrich bundle of rank $r$ on  $X^n_d$. Then, the Hilbert polynomial of $\cU$ looks as follows
\[
\chi(\cU(t))=\frac {r}{n!} \prod_{1 \le j \le n} (dj+t).
\]
In particular, evaluating at the above expression for $t=1$, we get:
\begin{equation} \label{divides}
n! \chi(\cU(1))= r \prod_{1 \le j \le n} (dj+1).
\end{equation}

\begin{lemma} \label{n!}
We have $n! \NN^* \subset \Ur(X^n_d)$. Moreover, let $m \in \NN^*$ be a divisor of $n!$ with $d \equiv m$ mod $n!$. Then $\Ur(X^n_d) \subset m\NN^*$. In particular, if $n!$ divides $d$, then $\Ur(X^n_d)=n! \NN^*$.
\end{lemma}

\begin{proof}
The first statement follows from the construction of \cite[Theorem 6.1]{eisenbud-schreyer:betti}.
Indeed, we consider the product $Y$ of $n$ copies of $\PP^1$ as a subvariety of degree $n!$ in $\PP^{2^n-1}$. Let $\pi$ be a generic projection of $Y$ to $\PP^n$, so that $\pi$ is a finite map of degree $n!$. Choose $s \in \NN^*$. Then, with obvious notation, $\cU = \pi_*(\cO_Y(d-1,\ldots,nd-1)^{\oplus s})$ is an Ulrich bundle on $X^n_d$ of rank $n!s$. Therefore $n! \NN^* \subset \Ur(X^n_d)$.

Moreover, consider an Ulrich bundle $\cU$ on $X^n_d$ of rank $r$ and a divisor $m$ of $n!$. 
Equation \eqref{divides} gives
\[
r \prod_{1 \le j \le n} (mj+1) \equiv 0, \qquad \mbox{mod $n!$,}
\]
since $d\equiv m$ mod $n!$.
Further reducing mod $m$ gives
$r \equiv 0$ mod $m$, which is what we want.
For $m=n!$, we get the last statement.
\end{proof}

\subsection{Coholomogy vanishing in the positive quadrant}

Let us consider the action of $\SL_2$ on $\PP^3$ by binary cubics, namely we identify $\PP^3$ with $\PP(V_3)$, in other words $\SL_2$ acts linearly on $\PP^3$ with the irreducible representation of weight $3$. We use the notion of \textit{instanton bundle} on $\PP^3$, which is an $H$-stable locally free sheaf $\cE$ of rank $r>0$, such that $c_1(\cE)=c_3(\cE)=0$ and $H^*(\cE(-2))=0$. Writing $c_2(\cE)=kH^2$, we say that $k$ is the \textit{charge} of $\cE$. Note that stability is not always part of the definition of instanton bundle, however we do require it here.
Also, note that $H^0(S^2 \cE)=0$ for any instanton bundle $\cE$ of rank $2$. Indeed,
$\cE$ is of rank $2$ with $c_1(\cE)=0$ so $\wedge^2 \cE \simeq \cO_{\PP^3}$. Also, $\cE$ is stable, hence simple, so $\Hom(\cE,\cE^\vee) \simeq \Hom(\cE,\cE)$ is one-dimensional, hence generated by the skew-symmetric duality responsible for the isomorphism $\wedge^2 \cE \simeq \cO_{\PP^3}$. So there is no non-zero symmetric map $\cE \to \cE^\vee$. 

Now, recall that according to \cite{faenzi:SL2}, for all $m \ge 1$ there is one and only one instanton bundle $\cE_m$ of rank $2$ on $\PP^3$ which is $\SL_2$-equivariant for the action of binary cubics and has charge ${m+1 \choose 2}$. Also, a rank-2 instanton bundle which is $\SL_2$-equivariant for the action of binary cubics is necessarily of the form $\cE_m$.

\begin{lemma}\label{coh0}
For $e \in \{0,1,2\}$ and $h \ge 0$, with $(h,e) \ne (0,0)$, write $d=6h+2e+1$, $m=3h+e$, $k = {m + 1 \choose 2}$.
Set $\cF_m = S^2 \cE_m$. Then we have $H^*(\cF_m(d-2))=0$.
\end{lemma}

\begin{proof}
We have $d=2m+1$ and a decomposition $\cE_m\otimes \cE_m$ as $\cF_m\oplus\cO_{\PP^3}$, so
\[
H^k(\cF_m(t)) = H^k(\cE_m\otimes \cE_m(t)), \qquad \mbox{for all $t \in \ZZ$ if $k \in \{1,2\}$.}
\]
We saw in the previous paragraph that $H^0(\cF_m)=0$, hence 
$H^0(\cF_m(-t))=0$ for any $t \ge 0$. By Serre duality, since $\cF_m$ is self-dual, we get $H^3(\cF_m(t-4))=0$ for any $t \ge 0$, in particular $H^3(\cF_m(d-2))=0$.

We write $\psi$ for the map $V_{3m} \otimes \cO_{\PP^3}(-m-1) \to V_{3m+1} \otimes \cO_{\PP^3}(-m)$ appearing in the resolution \eqref{resolution} of $\cE_m$ and consider such resolution as a two-step extension of $\cE_m$ and $\cO_{\PP^3}(-2m-1)$ given by an element
\[
\zeta \in \Ext^2_{\PP^3}(\cE_m,\cO_{\PP^3}(-2m-1)).
\]

For all integers $k$ and any locally free $\cG$, we look at the Yoneda map:
\[
\Upsilon_{\cG,k}:H^k(\cG \otimes \cE_m) \otimes \Ext^2_{\PP^3}(\cE_m,\cO_{\PP^3}(-2m-1)) \to H^{k+2}(\cG(-2m-1)).
\]
This is functorial with respect to morphisms $\lambda : \cG \to \cG'$, namely, these give rise to commutative diagrams:
\begin{equation} \label{yoneda_square}\xymatrix@-2ex{
H^k(\cG \otimes \cE_m) \otimes \Ext^2_{\PP^3}(\cE_m,\cO_{\PP^3}(-2m-1))  \ar^-{\Upsilon_{\cG,k}}[rr] \ar_{H^k(\lambda\otimes \cE_m)\otimes \mathrm{id}}[d] && H^{k+2}(\cG(-2m-1)) \ar^-{H^{k+2}(\lambda \otimes \cO_{\PP^3}(-2m-1))}[d] \\
H^k(\cG' \otimes \cE_m) \otimes \Ext^2_{\PP^3}(\cE_m,\cO_{\PP^3}(-2m-1)) \ar^-{\Upsilon_{\cG',k}}[rr] && H^{k+2}(\cG'(-2m-1)) 
}
\end{equation}
We again use the resolution \eqref{resolution} to observe that cupping with $\zeta$ induces isomorphisms:
\begin{align*}
\Upsilon_{\cO_{\PP^3}(m-2),1} (-\otimes \zeta) & : H^1(\cE_m(m-2)) \to H^3(\cO_{\PP^3}(-m-3)),\\
\Upsilon_{\cO_{\PP^3}(m-1),1} (-\otimes \zeta) & : H^1(\cE_m(m-1)) \to H^3(\cO_{\PP^3}(-m-2)).
\end{align*}
We denote both such isomorphisms by $\wedge \zeta$.

Consider the resolution of $\cE_m\otimes \cE_m(d-2)$ obtained by tensoring  \eqref{resolution} by $\cE_m(d-2)$, that is,
\[
    0 \rightarrow \cE_m(-2) \rightarrow V_{3m} \otimes \cE_m(m-2)\rightarrow V_{3m+1}\otimes \cE_m(m-1) \rightarrow \cE_m \otimes \cE_m(d-2) \rightarrow 0.
\]
We already know that $H^\ast(\cE_m(-2))=0$ since $\cE_m$ is an instanton bundle. 
Again from \eqref{resolution} we see
\[
H^k(\cE_m(m-1))=H^k(\cE_m(m-2))=0, \qquad \mbox{if $k \ne 1$}.
\]
Therefore $H^2(\cF_m(d-2))=H^2(\cE_m \otimes \cE_m(d-2))=0$.
Moreover we get a map
\[
f : V_{3m} \otimes  H^1(\cE_m(m-2)) \to V_{3m+1}\otimes  H^1(\cE_m(m-1)),
\]
obtained as $f=H^1(\psi \otimes \cE_m(2m-1))$,
with the property that
\begin{align} \label{property that}
        H^0(\cE_m \otimes \cE_m(d-2)) &\simeq \ker(f), &&
H^1(\cF_m(d-2)) = H^1(\cE_m \otimes \cE_m(d-2))   \simeq \coker(f).
\end{align}

Then applying \eqref{yoneda_square} to $\psi \otimes \cO_{\PP^3}(2m-2)$ yields a commutative diagram induced by cupping with $\zeta$: 
\[
\xymatrix{
V_{3m} \otimes H^1(\cE_m(m-2)) \ar_-{f=H^1(\psi \otimes \cE_m(2m-1))}[d] \ar^-{V_{3m} \otimes \wedge \zeta}[rr] && V_{3m} \otimes H^3(\cO_{\PP^3}(-m-3)) \ar^-{H^3(\psi \otimes \cO_{\PP^3}(-2))}[d]\\
V_{3m+1} \otimes H^1(\cE_m(m-1)) \ar_{V_{3m+1}\otimes \wedge \zeta}[rr] && V_{3m+1} \otimes H^3(\cO_{\PP^3}(-m-2))
}
\]

The horizontal arrows are isomorphisms and the vertical arrow on the right is surjective since $H^\ast(\cE_m(-2))=0$. Hence the vertical arrow on the left is also surjective. Also, the kernel of $H^3(\psi \otimes \cO_{\PP^3}(-2))$ is identified with $H^3(\cO_{\PP^3}(-2m-3))=H^3(\cO_{\PP^3}(-d-2))$. Therefore, \eqref{property that} gives:
\begin{align*}
H^1(\cF_m(d-2)) &= H^1(\cE_m\otimes\cE_m(d-2))=0 \\
H^0(\cF_m(d-2))\oplus H^0(\cO_{\PP^3}(d-2)) &= H^0(\cE_m\otimes \cE_m(d-2)) \simeq H^3(\cO_{\PP^3}(-d-2)).
\end{align*}

Hence $H^0(\cF_m(d-2)) = 0$. This finishes the proof that  $H^\ast(\cF_m(d-2)) = 0$.
\end{proof}

\subsection{The symmetric square as map to the space of orthogonal instantons}

Consider $\rO(3,4k)$ the coarse moduli space of rank-$3$ stable orthogonal bundles with Chern classes of the form $(0,4kH^2,0)$ and an element 
$([\cF],[\eta])$ of $\rO(3,4k)$. Then $\cF$ is a stable bundle, isomorphic to $\cF^\vee$ via $\eta$. If $\eta' : \cF \to \cF^\vee$ is not proportional to $\eta$, then in the pencil $\lambda \eta + \lambda' \eta'$, for $[\lambda:\lambda']\in \PP^1$, there must be a morphism $\cF \to \cF^\vee$ having trivial determinant. However, $\cF$ and $\cF^\vee$ are stable of the same Hilbert polynomial, so this cannot happen unless $\eta$ and $\eta'$ are proportional. So $[\eta]$ is actually determined uniquely by $[\cF]$ and we will suppress $[\eta]$ from the notation of the points in $\rO(3,4k)$. 

Also, since any $[\cE] \in \rI(2,k)$ is equipped with skew-symmetric duality $\eta_\cE$ which is canonical up to a non-zero scalar, $\cF = S^2 \cE$ is equipped with the canonical class $[\eta_\cF]$ of the symmetric duality $\eta_\cF=S^2 \eta_\cE$. Furthermore, the space $\rOI(3,4k)$ is an open subvariety of $\rO(3,4k)$. 

\begin{lemma}\label{etale}
Let $k \ge 1$ be an integer. Then the assignment $[\cE] \to [S^2 \cE]$ gives an étale map 
$$\varphi : \rI(2,k) \to \rO(3,4k),$$
which restricts to the morphism
$$\varphi : \rI(2,k) \setminus \rP(k) \to \rOI(3,4k).$$ 
So that, $\varphi(\rI(2,k))$ is a smooth irreducible Zariski-dense open subset of a component of $\rO(3,4k)$. 
\end{lemma}

\begin{proof}
Let $\cE$ be an instanton bundle in $\rI(2,k)$. We know from \cite{jardim-verbitsky:trihyperkahler} that $\rI(2,k)$ is smooth at $\cE$. More precisely, the tangent space and the obstruction space to $\rI(2,k)$ at the point $[\cE]$ corresponding to $\cE$ are described as:
\[
\cT_{[\cE]} \rI(2,k) = H^1(S^2 \cE), \qquad \mbox{with $h^1(S^2 \cE)=8k-3$}, \qquad 
\cO b_{[\cE]} \rI(2,k) = H^2(S^2 \cE)=0.
\]

Moreover $\cF=S^2\cE$ is slope-semistable since $\cE$ is and actually $\cF$ is slope-stable since $h^0(\cF)=h^0(\cF^\vee)=0$.
Therefore, computing Chern classes we see that $\cF$ is a bundle corresponding to a point $[\cF]$ of $\rO(3,4k)$.

Now recall that, for an orthogonal bundle $\cF$, the adjoint bundle $\ad(\cF)$ is defined over every fibre of $\cF$ by the adjoint representation
$\SO(3) \to \ad(\SO(3))$, see for instance \cite{ramanathan}. 
For a given $3$-dimensional vector space $F$, the associated adjoint representation is isomorphic to the exterior square $\wedge^2 F$, so $\ad(\cF) \simeq \wedge^2\cF$.
Moreover, since $\cF$ has rank $3$ and trivial determinant, we have
$\wedge^2\cF \simeq \cF^\vee$.  However $\cF$ being orthogonal, it is self-dual, namely $\cF \simeq \cF^\vee$.
Therefore, the tangent space and the obstruction space of $\rO(3,4k)$ at $[\cF]$ are described as
\[
\cT_{[\cF]} \rO(3,4k) = H^1(\ad(\cF))\simeq H^1(\cF), \qquad \cO b_{[\cF]} \rO(3,4k) = H^2(\ad(\cF)) \simeq H^2(\cF)
\]
Therefore we get the following information:
\begin{align*}
\cT_{[\cE]} \rI(2,k) & = H^1(S^2 \cE) \simeq H^1(\cF) = \cT_{[\cF]} \rO(3,4k), \\
0 = \cO b_{[\cE]} \rI(2,k) & = H^2(S^2 \cE) \simeq H^2(\wedge^2 \cF) = \cO b_{[\cF]} \rO(3,4k),
\end{align*}
where the isomorphism of tangent spaces is the differential of $\varphi$ at $[\cE]$. This proves that $\rO(3,4k)$ is smooth at $[\cF]$ and that $\varphi$ is a smooth morphism of relative dimension $0$, hence étale. Recall that $\rI(2,k)$ is irreducible by \cite{tikhomirov:odd, tikhomirov:even}. Then, the image of $\varphi$ is a smooth Zariski-dense subset of an irreducible component of $\rO(3,4k)$.

Now, assuming $[\cE]$ to lie away from the Le Potier divisor $\rP(k)$, by definition $\cF=S^2\cE$ verifies $H^*(\cF(-2))=0$ and we obtain the restriction of the image to $\rOI(3,4k)$.
\end{proof}

\subsection{Proof of the main theorems}

All that remains to finish the proof of Theorem \ref{main-2} is to combine the previous lemmas, the semicontinuity of the dimension of the cohomology groups and the result in \cite[Theorem 3.6]{costa-miro:instanton-ulrich} for the case of $r=3$. So that we just have to construct a rank $3$ instanton bundle of natural cohomology with charge $k={m + 1 \choose 2}$ for some integer $m\geq 1$ and $d=2m+1$.

\begin{proof}[Proof of Theorem \ref{main-2}]
The first part of the statement is done in Lemma \ref{etale}. To conclude we have to prove the existence of an Ulrich bundle of the form $S^2\cE(2(d-1))$ on $X_d$ for some $[\cE]$ sufficiently general in $\rI(2,k)$. For any $t \in \ZZ$, we consider the open subset $W_t$ of $\rO(3,4k)$ defined as
\[W_t=\{ [\cF] \in \rO(3,4k)  | H^\ast(\cF(t-2))=0\}.
\]

The intersection $U=W_d\cap W_{-d}\cap \varphi(\rI(2,k))$ is non-empty because the class of $\cF_m=S^2\cE_m$ belongs to $U$. Indeed, Lemma \ref{coh0} applies to guarantee that $[\cF_m] \in W_d$. Moreover, since $\cF_m$ is orthogonal, hence self-dual, using Serre duality we see that $[\cF_m]$ belongs to $W_{-d}$, that is:
\[
H^{j}(\cF_m(-d-2))^\vee \simeq H^{3-j}(\cF_m(d-2)) = 0, \qquad \mbox {for all $j \in \NN$.}
\] 

Then $U$ is an open subset of an irreducible component of $\rO(3,4k)$ by Lemma \ref{etale}. Furthermore, the intersection of $U$ with $\varphi(\rI(2,k)\setminus \rP(k))$ is also non-empty, being an intersection of non-empty open subsets inside the irreducible variety $\varphi(\rI(2,k))$ .

A representative of a class in the intersection $U \cap \varphi(\rI(2,k)\setminus \rP(k))$ is a rank $3$ instanton bundle $\cF$ with 
charge $(d^2-1)/{2}$, satisfying $H^\ast(\cF(-2-td))=0$ for $t \in \{-1,0,1\}$.
Therefore $\cF(2d-2)$ is an Ulrich bundle on $X_d$.
\end{proof}

\begin{proof}[Proof of Theorem \ref{main-1}] Let us separate the various four cases of $\bar{d}$.
Up to taking direct sums, it suffices to find Ulrich bundles of rank $2$ and/or $3$ in each case.
\begin{enumerate}[label=\roman*)]
\item This case is proven in Lemma \ref{n!}.
\item When $\bar{d} \in \{1,5\}$, the value $e$  appearing in the decomposition of $d$ of Theorem \ref{main-2} satisfies $e\in \{0,2\}$. From Theorem \ref{main-2} and \cite[Proposition 3.7]{costa-miro:instanton-ulrich}, we get, respectively, the existence of Ulrich bundles of rank $3$ and $2$ and consequently of any rank $r\ge 2$. The non-existence of Ulrich line bundles is a consequence of $\Pic(X_d)=\mathbb{Z}$ and $d\geq 2$.
\item If $\bar{d} \in \{2,4\}$, then $2|r$ by Lemma \ref{n!}. Therefore, we only need to apply \cite[Proposition 3.7]{costa-miro:instanton-ulrich} to conclude this case.
\item For $\bar{d}=3$, this corresponds to the case $e=1$ in Theorem \ref{main-2}. We get $3|r$ by Lemma \ref{n!}. Applying Theorem \ref{main-2} affords the existence of Ulrich bundles of rank $3$ and thereby the equality $\Ur(X_d)=3\NN^*$.
\end{enumerate}
\end{proof}

\bibliographystyle{amsalpha.v2}
\bibliography{ulrich.bib}

\providecommand{\bysame}{\leavevmode\hbox to3em{\hrulefill}\thinspace}
\providecommand{\MR}{\relax\ifhmode\unskip\space\fi MR }
\providecommand{\MRhref}[2]{%
  \href{http://www.ams.org/mathscinet-getitem?mr=#1}{#2}
}
\providecommand{\href}[2]{#2}
\begin{thebibliography}{CMRPL21}

\bibitem[AFO17]{farkas-aprodu-ortega}
Marian Aprodu, Gavril Farkas, and Angela Ortega, \emph{Minimal resolutions,
  {C}how forms and {U}lrich bundles on {$K3$} surfaces}, J. Reine Angew. Math.
  \textbf{730} (2017), 225--249.

\bibitem[BES17]{blaser-eisenbud-schreyer}
Markus Bl\"{a}ser, David Eisenbud, and Frank-Olaf Schreyer, \emph{Ulrich
  complexity}, Differential Geom. Appl. \textbf{55} (2017), 128--145.

\bibitem[BS97]{bor-segert}
Gil Bor and Jan Segert, \emph{Symmetric instantons and the {ADHM}
  construction}, Comm. Math. Phys. \textbf{183} (1997), no.~1, 183--203.

\bibitem[BFL22]{boralevi-faenzi-lella}
Ada Boralevi, Daniele Faenzi, and Paolo Lella, \emph{A construction of
  equivariant bundles on the space of symmetric forms}, Rev. Mat. Iberoam.
  \textbf{38} (2022), no.~3, 761--782.

\bibitem[CFK23]{ciliberto-flamini-knutsen:del_pezzo}
Ciro Ciliberto, Flaminio Flamini, and Andreas~Leopold Knutsen, \emph{Ulrich
  bundles on {D}el {P}ezzo threefolds}, J. Algebra \textbf{634} (2023),
  209--236.

\bibitem[CMR18]{costa-miro:ulrich-veronese-surfaces}
Laura Costa and Rosa~Maria Mir\'{o}-Roig, \emph{Ulrich bundles on {V}eronese
  surfaces}, Singularities, algebraic geometry, commutative algebra, and
  related topics, Springer, Cham, 2018, pp.~375--381.

\bibitem[CMR21]{costa-miro:instanton-ulrich}
\bysame, \emph{Instanton bundles vs {U}lrich bundles on projective spaces},
  Beitr. Algebra Geom. \textbf{62} (2021), no.~2, 429--439.

\bibitem[CMRPL21]{costa-miro-llopis}
Laura Costa, Rosa~Mar\'{\i}a Mir\'{o}-Roig, and Joan Pons-Llopis, \emph{Ulrich
  bundles---from commutative algebra to algebraic geometry}, De Gruyter Studies
  in Mathematics, vol.~77, De Gruyter, Berlin, [2021] \copyright 2021.

\bibitem[ES03]{eisenbud-schreyer:chow}
David Eisenbud and Frank-Olaf Schreyer, \emph{Resultants and {C}how forms via
  exterior syzygies}, J. Amer. Math. Soc. \textbf{16} (2003), no.~3, 537--579,
  With an appendix by Jerzy Weyman.

\bibitem[ES09]{eisenbud-schreyer:betti}
\bysame, \emph{Betti numbers of graded modules and cohomology of vector
  bundles}, J. Amer. Math. Soc. \textbf{22} (2009), no.~3, 859--888.

\bibitem[Fae07]{faenzi:SL2}
Daniele Faenzi, \emph{Homogeneous instanton bundles on {$\Bbb P^3$} for the
  action of {${\rm SL}(2)$}}, J. Geom. Phys. \textbf{57} (2007), no.~10,
  2146--2157.

\bibitem[Fae19]{K3}
\bysame, \emph{Ulrich bundles on {K}3 surfaces}, Algebra Number Theory
  \textbf{13} (2019), no.~6, 1443--1454.

\bibitem[HH82]{hartshorne-hirschiwitz}
Robin Hartshorne and Andr\'{e} Hirschowitz, \emph{Cohomology of a general
  instanton bundle}, Ann. Sci. \'{E}cole Norm. Sup. (4) \textbf{15} (1982),
  no.~2, 365--390.

\bibitem[JV14]{jardim-verbitsky:trihyperkahler}
Marcos Jardim and Misha Verbitsky, \emph{Trihyperk\"{a}hler reduction and
  instanton bundles on {$\Bbb{C}\Bbb{P}^3$}}, Compos. Math. \textbf{150}
  (2014), no.~11, 1836--1868.

\bibitem[LP83]{le_potier:yang-mills}
Joseph Le~Potier, \emph{Sur l'espace de modules des fibr\'{e}s de {Y}ang et
  {M}ills}, Mathematics and physics ({P}aris, 1979/1982), Progr. Math.,
  vol.~37, Birkh\"{a}user Boston, Boston, MA, 1983, pp.~65--137.

\bibitem[LR24]{lopez-raychaudhury:veronese}
Angelo~Felice Lopez and Debaditya Raychaudhury, \emph{Non-existence of low rank
  {U}lrich bundles on {V}eronese varieties}, In preparation, available at {\tt
  http://ricerca.matfis.uniroma3.it//users/lopez/Veronese.pdf}, 2024.

\bibitem[Ram75]{ramanathan}
Annamalai Ramanathan, \emph{Stable principal bundles on a compact {R}iemann
  surface}, Math. Ann. \textbf{213} (1975), 129--152.

\bibitem[Tik12]{tikhomirov:odd}
Alexander~S. Tikhomirov, \emph{Moduli of mathematical instanton vector bundles
  with odd {$c_2$} on projective space}, Izv. Ross. Akad. Nauk Ser. Mat.
  \textbf{76} (2012), no.~5, 143--224.

\bibitem[Tik13]{tikhomirov:even}
\bysame, \emph{Moduli of mathematical instanton vector bundles with even
  {$c_2$} on projective space}, Izv. Ross. Akad. Nauk Ser. Mat. \textbf{77}
  (2013), no.~6, 139--168.

\end{thebibliography}

\end{document}